\documentclass[bezier,amstex, oneside,reqno]{amsart}
\usepackage{amsmath, amssymb, amsthm, verbatim, euscript, amscd, textcomp, enumerate}
\usepackage[dvips]{graphicx}
 \usepackage[all,cmtip]{xy}
\usepackage{epsfig}
\usepackage{color}

\def\ie{\emph{i.e., }}
\def\eg{\emph{e.g., }}
\def\B{\mathcal B}
\def\G{\mathcal G}

\def\S{\mathbb S}
\def\DD{\mathbb D}

\def\R{\mathbb R}
\def\Z{\mathbb Z}
\def\Q{\mathbb Q}
\def\U{\mathcal U}
\def\V{\mathcal V}

\def\C{\mathbb C}
\def\T{\mathbb T}

\def\E{\EuScript E}

\def\A{\mathcal A}

\def\e{\varepsilon}

\newtheorem{theorem}{Theorem}[section]
\newtheorem{prop}[theorem]{Proposition}
\newtheorem{conj}[theorem]{Conjecture}
\newtheorem{q}[theorem]{Question}
\newtheorem{add}[theorem]{Addendum}
\newtheorem{claim}[theorem]{Claim}

\newtheorem{cor}[theorem]{Corollary}
\newtheorem{lemma}[theorem]{Lemma}
\newtheorem{n}[theorem]{}
 \theoremstyle{remark}
\newtheorem{remark}[theorem]{Remark}

\theoremstyle{remark}

\makeatletter
\renewcommand*\env@matrix[1][*\c@MaxMatrixCols c]{%
  \hskip -\arraycolsep
  \let\@ifnextchar\new@ifnextchar
  \array{#1}}
\makeatother

\begin{document}
\author{Andrey Gogolev$^\ast$, Pedro Ontaneda$^{\ast\ast}$ and Federico Rodriguez Hertz$^{\ast\ast\ast}$}
\title[Partially hyperbolic diffeomorphisms]{New partially hyperbolic dynamical systems I}
\thanks{$^\ast$The first author was partially supported by NSF grant DMS-1266282. He also would like to acknowledge excellent working environment provided by the Institute for Math at Stony Brook University.\\
$^{\ast\ast}$ The second author  was partially supported by NSF grant DMS-1206622.\\
$^{\ast\ast\ast}$The last author  was partially supported by NSF grant DMS-1201326.
}
\begin{abstract}
We propose a new method for constructing partially hyperbolic diffeomorphisms on closed manifolds. As a demonstration of the method we show that there are simply connected closed manifolds that support partially hyperbolic diffeomorphisms. { Laying aside many surgery constructions of Anosov flows (mostly in dimension three), these are the first new examples of manifolds which admit partially hyperbolic diffeomorphisms in the past 40 years.}
\end{abstract}
\date{}
 \maketitle

\section{Introduction}

Let $M$ be a  smooth compact $d$-dimensional manifold.
A diffeomorphism $F$ is called {\it Anosov}  if there
exist a constant $\lambda>1$ and a Riemannian metric along with a
$DF$-invariant splitting $TM=E^s\oplus E^u$ of the tangent bundle of
$M$, such that for any unit vectors, $v^s$ and $v^u$ in $E^s$ and $E^u$, respectively, we have 
\[\arraycolsep=1.4pt\def\arraystretch{1.6}
\begin{array}{rcccl}
& &\|DF(v^s)\|&\le& \lambda^{-1}\;\;\;\\
\lambda&\le &\|DF(v^u)\|&
\end{array}
\]
All known examples of Anosov diffeomorphisms are supported on manifolds which are homeomorphic to infranilmanifolds. The classification problem for Anosov diffeomorphisms is an outstanding open problem that goes back to Anosov and Smale. The great success of the theory of Anosov diffeomorphisms (and flows)~\cite{A}  motivated Hirsch-Pugh-Shub~\cite{HPS2, HPS} and Brin-Pesin~\cite{BP} to relax the definition as follows.

A diffeomorphism $F$ is called {\it partially hyperbolic} if there
exist a constant $\lambda>1$ and a Riemannian metric along with a
$DF$-invariant splitting $TM=E^s\oplus E^c\oplus E^u$ of the tangent bundle of
$M$, such that for any unit vectors, $v^s,v^c, v^u$ in $E^s, E^c, E^u$, respectively, we have 
\[\arraycolsep=1.4pt\def\arraystretch{1.6}
\begin{array}{rcccl}
& &\|DF(v^s)\|&\le& \lambda^{-1}\;\;\;\\
\|DF(v^s)\|&<& \|DF(v^c)\|&<&\|DF(v^u)\|  \\
\lambda&\le &\|DF(v^u)\|&
\end{array}
\]
In recent years the dynamics of partially hyperbolic diffeomorphisms has been a popular subject, see, \eg~\cite{PS, RHRHU}. The pool of examples of partially hyperbolic diffeomorphisms is larger than that of Anosov diffeomorphisms, in particular, due to the fact that extensions (\eg $F\times id_N$) of partially hyperbolic diffeomorphisms are partially hyperbolic. However, the collection of basic ``building blocks" for partially hyperbolic diffeomorphisms is still rather limited. Up to homotopy, all previously known examples of irreducible\footnote{See Section~\ref{sec_irreducible} for our definition of ``irreducible"} partially hyperbolic diffeomorphisms are either affine diffeomorphisms on homogeneous spaces or time-1 maps of Anosov flows. The affine examples go back to Brin-Pesin~\cite{BP} and Sacksteder~\cite{S}.
\begin{theorem}[Main Theorem]\label{theorem_main}
 For any $d\ge 6$ there exist a closed $d$-dimensional simply connected manifold $M$ that supports a volume preserving partially hyperbolic diffeomorphism $F\colon M\to M$. Moreover, $F$ is ergodic with respect to volume.
\end{theorem}
\begin{remark}
There are no previously known examples of partially hyperbolic diffeomorphisms on simply connected manifolds. It is easy to show that simply connected compact Lie groups do not admit partially hyperbolic automorphisms (use, \eg~\cite[Theorems~6.61, 6.63]{HS13}). However, to the best of our knowledge, the possibility that some simply connected manifolds support Anosov flows is open.
\end{remark}
Burago and Ivanov proved that simply connected 3-manifolds (\ie the sphere $\mathbb S^3$) do not support partially hyperbolic diffeomorphisms~\cite{BI}. Simply connected 4-manifolds have non-zero Euler characteristic and hence do not admit line fields. Consequently simply connected 4-manifolds do not support partially hyperbolic diffeomorphisms. 
\begin{q}
 Do simply connected 5-manifolds support partially hyperbolic diffeomorphisms?
\end{q}
\begin{remark}
{
It is easy to see, for topological reasons, that the 5-sphere $\mathbb S^5$ does not admit partially hyperbolic diffeomorphisms. Indeed, the splitting $T\mathbb S^5=E^s\oplus E^c\oplus E^u$ is either a 3-1-1 or a 2-2-1 splitting. By an old result of Eckmann~\cite{Eck} (see also Whitehead~\cite{Wh}) the sphere $\mathbb S^5$ does not admit two linearly independent vector fields and, therefore, the splitting must be a 2-2-1 splitting.  Because the structure group $GL^+(2,\R)$ retracts to $O(2,\R)$, the 2-plane bundles over $\mathbb S^5$ are in one-to-one correspondence with circle bundles, which are classified by homotopy classes of maps $\mathbb S^4\to\textup{Diff}(S^1)$ by the clutching construction. But $\textup{Diff}(S^1)$ is homotopy equivalent to $S^1$ and, hence, all 2-plane bundles over $\mathbb S^5$ are trivial. Therefore the existence of a 2-2-1 splitting of $T\mathbb S^5$ implies that $\mathbb S^5$ is parallelizable, which is a contradiction. It is an interesting open problem to decide whether $\mathbb S^7$ supports partially hyperbolic diffeomorphisms.
}
\end{remark}

In the next section we briefly (and very informally) outline our approach. Then we proceed with a detailed discussion leading to the proof of the Main Theorem in Section~\ref{sec_proof}.
The authors would like to thank the referee for his/her careful reading.

\section{Informal description of the construction}

Our approach is to consider a smooth fiber bundle $M\to E\stackrel{p}{\to} X$, whose base $X$ is a closed manifold and whose fiber $M$ is  a closed manifold which admits a partially hyperbolic diffeomorphism. The idea now is to equip the total space $E$ with a fiberwise partially hyperbolic diffeomorphism $F\colon E\to E$, which fibers over a diffeomorphism $f\colon X\to X$, \ie the following diagram commutes
$$
\xymatrix{
E\ar_p[d]\ar^F[r] & E\ar_p[d]\\
X\ar^f[r] & X
}
$$
Then the diffeomorphism $F$ is partially hyperbolic provided that $f$ is dominated by the action (on extremal subbundles) of $F$ along the fibers. However, for non-trivial fiber bundles the bundle map $p\colon E\to X$ intertwines the dynamics in the fiber with dynamics in the base, which makes it difficult to satisfy
\renewcommand\labelenumi{\theenumi.}
\begin{enumerate}
\item $F$ is fiberwise partially hyperbolic;
\item $f$ is dominated by $F$;
\end{enumerate}
at the same time. In particular, if $X$ is simply connected and $f$ is homotopic to $id_X$ such constructions seem to be out of reach (cf.~\cite[Question~6.5]{FG}). Moreover, assuming that $f=id_X$, it was shown in~\cite{FG} that such construction is, in fact, impossible in certain more restrictive setups. However, in this paper, we show that if $f^*\colon H^*(X)\to H^*(X)$ is allowed to be non-trivial then  our method works in the setup of principal torus bundles over simply connected 4-manifolds.

\section{Preliminaries on principal bundles}
\label{sec_sec2}
\label{sec_prelim_bundles}
In this section we review some of the concepts and facts
about principal fiber bundles that will be needed later. For more
details consult~\cite{Hus}.

{\bf Standing assumption.} {\it In this and further sections we will always assume that all topological spaces are connected countable CW complexes. { Given a space $X$ we will write $H_*(X; A)$ and $H^*(X; A)$ for its homology and cohomology groups with coefficients in an abelian group $A$. If we abbreviate to $H_*(X)$ and $H^*(X)$ then we assume that the coefficient group is $\Z$. }}

Let $X$ be a space and $G$ be a topological group. Recall that a (locally trivial)
principal $G$-bundle $\pi\colon E\rightarrow X$ is a (locally trivial)
fiber bundle with fiber $G$ and structure group $G$, where
(the group) $G$ acts on (the fiber) $G$ by left multiplication.
Let $\phi_\alpha\colon U_\alpha\times G\rightarrow \pi^{-1}U_\alpha$, $\alpha\in\A$,
 be a complete collection of trivializing
charts of the principal bundle $E$. Denote by $U_{\alpha \beta}$ the intersection $U_{\alpha}\cap U_{ \beta}$, $\alpha,\beta\in\A$.
Also define $\phi_{\alpha \beta}\colon U_{\alpha \beta}\rightarrow G$ in the following way
 $$(\phi_{ \beta}^{-1}\circ\phi_{\alpha})(x,g)=
(x,\phi_{\alpha \beta}(x)\cdot g),\;\;\; x\in U_{\alpha \beta}\neq\varnothing, \; g\in G.
$$
 This  collection of {\it transition functions} $\{\phi_{\alpha \beta}\}_{U_{\alpha \beta}\neq\varnothing}$ satisfies the following {\it cocycle condition} 
$$
\phi_{\alpha \beta}(x)\cdot\phi_{ \beta\gamma}(x)\cdot\phi_{\gamma\alpha}(x)=e, \;\;\; x\in U_{\alpha}\cap U_\beta\cap U_\gamma,
$$
 where $e$ is the identity in $G$.

Conversely, let $\{ \phi_{\alpha \beta} \}_{U_{\alpha\beta}\neq\varnothing}$ be a {\it cocycle of transition functions over a covering $\{ U_\alpha\}$}; that is, assume that we have

\begin{enumerate}
\item an open covering $\{ U_\alpha\}$ of the space $X$;
\item
a collection of maps $\phi_{\alpha \beta}\colon U_{\alpha \beta}\rightarrow G$,
$U_{\alpha \beta}\neq\varnothing$, that satisfy the cocycle condition. 
\end{enumerate}

Then we can construct a principal $G$-bundle $E\rightarrow  X$
by gluing the spaces $U_\alpha\times G$ using the transition functions
$\{\phi_{\alpha \beta}\}$. The cocycle condition
ensures that the gluings are consistent.

We will need the following facts:

\begin{enumerate}[({\ref{sec_sec2}.}1)]
 
\item \label{fact21} Every principal $G$-bundle $E\rightarrow X$ has a (right) action $E\times G\rightarrow E$. This action is free
and the orbits are exactly the fibers. This can be seen from
the construction of $E$ using a cocycle of transition functions:
define $\phi_\alpha (x,g). h=\phi_\alpha(x,g.h)$. This is
well defined because right and left translations on $G$
commute. (There are other equivalent ways of defining
principal bundles. In some of them the action is
included in the definition.)

\item { Two $G$-bundles $\pi\colon E\to X$ and $\pi'\colon E'\to X$ are called {\it equivalent} and we write $E\cong E'$ if there exists a homeomorphism $f\colon E\to E'$ which fits into the commutative diagram
$$
\xymatrix{
E \ar_{\pi}[rd]\ar^f[rr] & & E' \\
 & X\ar_{\pi'}[ru] &
}
$$
and commutes with the $G$-action, \ie $f(y.g)=f(y).g$.
}
\item From~(\ref{sec_sec2}.\ref{fact21})  we get a canonical (up to right
translation) way of identifying a fiber of a principal $G$-bundle
with $G$.

\item \label{fact23} For every $G$ there is a principal $G$-bundle
$EG\rightarrow BG$ such that for any space $X$ and any principal $G$-bundle $E\rightarrow X$,  there is a unique, up to homotopy, map $\rho\colon X\rightarrow BG$,
such that $E\cong \rho^*EG$. { We say that map $\rho$ {\it classifies} the bundle $E\rightarrow X$. } The space $BG$ is called {\it the classifying
space} of $G$, and the $G$-bundle $EG\rightarrow BG$ is called {\it the universal
principal $G$-bundle.}

\item The classifying space of the topological group $\S^1$ is $\C P^\infty=\cup _{n\ge 0} \C P^n$. The universal
principal $\S^1$-bundle is $E\S^1=\S^\infty\rightarrow
\C P^\infty$. Here $\S^\infty=\cup_{n\ge 0}\S^n$.
This bundle is the limit of $\S^1\rightarrow\S^{2n-1}\rightarrow
\C P^n$, where $\S^1\subset\C$ acts on $\S^{2n-1}\subset\C^n$
by scalar multiplication.

\item \label{fact25}  We have  $B(G\times H)=BG\times BH$, provided that both $BG$ and $BH$ are countable $CW$ complexes. 
Moreover, $E(G\times H)=EG\times EH$ and the action and
projections respect the product structure.
It follows that $B\T^k=\underbrace{\C P^\infty\times...\C P^\infty}_k$, where $\T^k=\underbrace{\S^1\times...\times\S^1}_k$ is the $k$-torus, and $E\T^k=\underbrace{\S^\infty\times...\times \S^\infty}_k$.
\end{enumerate}
\renewcommand\labelenumi{\theenumi.}

\section{$A(E)$ construction}

Let $\pi_1\colon E_1\rightarrow X_1$ and $\pi_2\colon E_2\rightarrow X_2$ be principal
$G$-bundles. A fiber preserving map $F\colon E_1\rightarrow E_2$,
covering $f\colon X_1\rightarrow X_2$ (i.e., $f\circ \pi_1=\pi_2\circ F$)
is a {\it principal $G$-bundle map} if $F$ commutes with the right
action of $G$, that is, $F(y.g)=F(y).g$, $y\in E_1$, $g\in G$.
Hence $F$ restricted to a fiber is a left translation.

More generally, let $A\colon G\rightarrow G$ be an automorphism
of the topological group $G$ and let $E_1$, $E_2$ be as above.
We say that a map $F\colon E_1\rightarrow E_2$,
covering $f\colon X_1\rightarrow X_2$ is an {\it $A$-bundle map}
(or simply an {\it $A$-map}) if $F(y.g)=F(y).A(g)$ for all $y\in E_1$, $g\in G$.
Hence $F$ restricted to a fiber is the automorphism $A$ composed with a left translation.

\begin{remark} \label{rem_equivalence}
Of course, an $id_G$-map is just a principal $G$-bundle
map. In particular, an $id_G$-map that covers the identity
$id_X\colon X\rightarrow X$ is a principal $G$-bundle equivalence.
\end{remark}
\begin{remark}\label{rem_composition}
Note that the composition of an
$A$-map and a $B$-map is a $BA$-map.
\end{remark}

Now let $\pi\colon E\rightarrow X$ be a principal $G$-bundle, and
let $\{\phi_{\alpha \beta}\}$ be a cocycle of transition functions
for $E$. Note that  $\{A\circ \phi_{\alpha \beta}\}$ is also
a cocycle of transition functions. This is because
$$
A(\phi_{\alpha \beta}(x))\cdot A(\phi_{ \beta\gamma}(x))\cdot A(\phi_{\gamma\alpha}(x))=A\Big(\phi_{\alpha \beta}(x)\cdot \phi_{ \beta\gamma}(x)\cdot\phi_{\gamma\alpha}(x)\Big)=A(e)=e.
$$
Therefore the new cocycle of transition functions $\{A\circ \phi_{\alpha \beta}\}$
defines a principal $G$-bundle over $X$. We denote this bundle
by $A(E)$. Next we show that $A(E)$ is well defined.

\begin{prop} The principal $G$-bundle $A(E)$ does not
depend on the choice of the cocycle of transition functions $\{ \phi_{\alpha \beta}\}$.
\end{prop}

\begin{proof} Let $\{ \phi_{\alpha \beta}\}$, over
the covering $\{U_\alpha\}$, and $\{ \psi_{a b}\}$, over
the covering $\{V_a\}$, be two cocycles of transition functions, both defining
equivalent principal $G$-bundles. Denote the corresponding bundles
by $E$ and $E'$, respectively.

\medskip
\noindent {\bf Special case.} {\it The cocycle $\{ \psi_{a b}\}$
is a refinement of $\{ \phi_{\alpha \beta}\}$. That is,
the covering $\{V_a\}$ is a refinement of $\{U_\alpha\}$
(i.e every $V_a$ is contained in some $U_\alpha$), and
every $\psi_{ab}$ is the restriction of some $\phi_{\alpha \beta}$.}

\medskip
Recall that in this case the principal bundle equivalence
between $E$ and $E'$ is simply given by inclusions:
the element $(x,g)\in V_a\times G$ maps to $(x,g)\in U_\alpha\times G$, where $U_\alpha$ is a fixed (for each $a$)
element of $\{U_\alpha\}$ such that $V_a\subset U_\alpha$.

It is straightforward to verify that the same rule defines
an equivalence between $\{A\circ \phi_{\alpha \beta}\}$ and $\{A\circ \psi_{\alpha \beta}\}$. This proves the
special case.

Because of the special case we can now assume that both
cocycles $\{ \phi_{\alpha \beta}\}$, $\{ \psi_{a b}\}$
are defined over the same covering $\{U_\alpha\}$.
Then the existence of a principal bundle equivalence
between $E$ and $E'$ is equivalent to the existence a collection of functions $\{r_\alpha\}$,
$r_\alpha\colon U_\alpha\rightarrow G$ such that
\begin{equation}\phi_{\alpha\beta}(x)\cdot r_\alpha(x)=r_\beta(x)\cdot\psi_{\alpha\beta}(x)\label{eq_1}\end{equation}
 for $x\in U_{\alpha\beta}$ (see~\cite[Chapter 5, Theorem 2.7]{Hus}).
Applying $A$ to equation~(\ref{eq_1}) we obtain
$$\big(A\circ\phi_{\alpha\beta}\big)(x)\cdot\big(A\circ r_\alpha\big)(x)=\big( A\circ r_\beta\big)(x)\cdot\big(A\circ \psi_{\alpha\beta}\big)(x).$$
 Therefore, the collection $\{A\circ r_\alpha\}$ defines
a principal bundle equivalence between $\{A\circ \phi_{\alpha \beta}\}$ and $\{A\circ \psi_{\alpha \beta}\}$.
\end{proof}

\begin{prop} Let $E\rightarrow X$ be a principal $G$-bundle. Also let
$A$ and $B$ be automorphisms of $G$. Then
$$(AB)(E)=A(B(E))\,\,\,\,\,\,\, and \,\,\,\,\,\,\, id_G(E)=E.$$\end{prop}

\begin{proof} Direct from the definition of $A(E)$.
\end{proof}

\begin{prop} 
\label{prop_commute}
Let $E\rightarrow X$ be a principal $G$-bundle, let
$A$ an automorphism of $G$ and let $f\colon Z\rightarrow X$ be a map. Then
$$ f^*\big(A(E)\big)= A\big(f^*(E)\big)$$\end{prop}

\begin{proof} Let $\{ \phi_{\alpha \beta}\}$ be a cocycle of transition functions
for $E$ defined  over
a covering $\{U_\alpha\}$. Then $\{ A\circ\phi_{\alpha \beta}\circ f\}$ is cocycle of transition functions
over $\{f^{-1}U_\alpha\}$ for both $f^*\big(A(E)\big)$ and  $A\big(f^*(E)\big)$.
\end{proof}

\begin{prop}\label{prop_FA}
Let $E\rightarrow X$ be a principal $G$-bundle
and let $A$ be an automorphism of $G$. Then there is an $A$-map
$F_A\colon E\rightarrow A(E)$, covering the identity $id_X\colon X\rightarrow X$.
\end{prop}

\begin{proof} Let $\{ \phi_{\alpha \beta}\}$ be a cocycle of transition functions
for $E$ over a covering $\{U_\alpha\}$. Then $\{A\circ \phi_{\alpha \beta}\}$ is a cocycle of transition functions
for $A(E)$  over $\{U_\alpha\}$. Define map $F_A$ in  charts as follows: 
$$
U_\alpha\times G\ni (x,g)\mapsto (x,A(g))\in U_\alpha\times G,
$$
where the latter copy of $U_\alpha\times G$ is a chart of
$A(E)$. The map $F_A$ is well defined because the following
diagram commutes
$$
\xymatrixcolsep{5pc}
\xymatrix{
G\ar_A[d]\ar^{L_{\phi_{\alpha \beta}(x)}}[r] & G\ar_A[d] \\
G\ar^{L_{A(\phi_{\alpha \beta}(x))}}[r] & G
}
$$
 Here $L_h$ denotes left multiplication by $h$.
\end{proof}

\begin{cor} \label{cor_FA_inverse}
Let $E\rightarrow X$ be a principal $G$-bundle
and let $A$ be an automorphism of $G$. Then there is an $A$-map
$F_{A^{-1}}\colon A^{-1}(E)\rightarrow E$, covering the identity $id_X\colon X\rightarrow X$.
\end{cor}
\begin{proof}
 This follows from Propositions~\ref{prop_FA} and Remark~\ref{rem_composition}.
\end{proof}

Let $EG\rightarrow BG$ be the universal principal $G$-bundle
and let $A$ be an automorphism of $G$. Then $A(EG)$ is a 
principal $G$-bundle, hence (see (\ref{sec_sec2}.\ref{fact23})) there is a map
$\rho_A\colon BG\rightarrow BG$ such that 
\begin{equation}\label{eq_rho}
A(EG)\cong\rho_A^*(EG).
\end{equation}
Moreover, this map is unique up to homotopy.

\section{Principal $\T^k$-Bundles}

We now take $G=\T^k=\S^1\times\ldots\times\S^1$.  Recall that by~(\ref{sec_sec2}.\ref{fact25}) $B\T^k=(\C P^\infty)^k$. Therefore $\pi_2 B\T^k$ is canonically identified with $\Z^k$ (by identifying $i$-th generator of $\Z^k$ with the canonical generator of the second homotopy group of the $i$-th copy of $\C P^\infty$.).

Let $A\in SL (\Z,k)$. The matrix $A$ induces automorphisms $A\colon \Z^k\to \Z^k$ and $A\colon\T^k\to\T^k$ for which we use the same notation.




The next proposition is a key result and its proof occupies the rest of this section (except for the lemma at the end of this section). Recall that $\rho_A$ is characterized by equation~(\ref{eq_rho}).

\begin{prop} 
\label{prop_key}
Let $g\colon B\T^k\rightarrow B\T^k$ be a map such that
$\pi_2(g)=A\in SL(\Z,k)$. Then
$g$ is homotopic to $\rho_A$,
that is,
$$A(E\T^k)\cong g^*(E\T^k) .$$
\end{prop}

\begin{proof}
The proof will require some lemmas and claims.

We consider $\C P^\infty=\cup_n\C P^n$
with the usual $CW$-structure, \ie one cell in each even dimension.
This structure induces a product $CW$-structure on $B\T^k=(\C P^\infty)^k$. Then the 2-skeleton of $B\T^k$ is
the wedge $\bigvee_{i=1}^k\S_i^2$ of $k$ copies of the 2-sphere
$\S^2$. Denote by $Y$ this 2-skeleton and by $E\to Y$ the restriction
of $E\T^k\to B\T^k$ to $Y$.  We first prove the proposition for the principal $\T^k$-bundle
$E\rightarrow Y$.

\begin{lemma} \label{lemma_key}
Let $g_Y\colon Y\to Y$ be a map such that
$\pi_2(g_Y)=A\in SL(\Z,k)$. Then
$$A(E)\cong g_Y^*(E).  $$
\end{lemma}

\begin{proof} Let $p$ be the wedge point of $Y$. Then we have $\S^2_i\cap\S^2_j=\{ p\}$, $i\neq j$. We identify $p$ with the south pole of
each $\S^2_i$. Denote by $D_i^+$ and $D_i^-$ the closed upper
and lower hemispheres of $\S^2_i$, respectively.

Let $E_i\to \S_i$ be the restriction of $E\to Y$ to $\S_i$, $i=1,\ldots k$.
\begin{claim}\label{claim_43}
The principal $\T^k$-bundle $E_i\to \S_i$ is obtained by identifying
$D^-_i\times \T^k$ with $D^+_i\times\T^k$ along their boundaries
using the gluing map $\omega_i\colon\S^1\rightarrow \T^k$,
$\omega_i(u)=(1,...,1,u,1,...1)$, that is, all coordinates of 
$\omega_i(u)$ are equal to $1\in\S^1$, except for the $i$-th
coordinate, which is equal to $u$.
\end{claim}

\begin{proof}
 The claim follows from
putting together the following two facts; see also~(\ref{sec_sec2}.\ref{fact25}).

\begin{enumerate}
\item The 2-skeleton of $B\S^1=\C P^\infty$
is $\C P^1=\S^2$, and the restriction of
$E\S^1=\S^\infty$ to $\S^2$ is the Hopf bundle
$\S^1\rightarrow \S^3\rightarrow\S^2$. Moreover,
$\S^3$ is obtained by identifying two copies
of $\DD^2\times\S^1$ along the boundaries using the
identity map $id_{\S^1}\colon\S^1\rightarrow \S^1$ as gluing map.

\item Let $F_1\rightarrow E_1\rightarrow X_1$
and $F_2\rightarrow E_2\rightarrow X_2$ be
two fiber bundles. Consider the inclusion  $X_1\hookrightarrow X_1\times X_2$, $x\mapsto (x,*)$, for some fixed $*\in X_2$.
Then the restriction $(E_1\times E_2)|_{X_1}$ of the product bundle $F_1\times F_2\rightarrow E_1\times E_2\rightarrow X_1\times X_2$ to $X_1\subset X_1\times X_2$ is the bundle
$F_1\times F_2\rightarrow E_1\times F_2 \rightarrow X_1$.
\end{enumerate}
\end{proof}
Write $A=(a_{ij})\in SL(\Z,k)$. Because $A=\pi_2(g_Y)$,
after performing a homotopy, we can assume
that $g_Y$ satisfies the following property.

\begin{n}\label{g_property} For each $j$ there are $k$ disjoint closed 2-disks
$D_{ij }\subset D_j^+$, $i=1,...,k$, such that
\begin{enumerate}
\item $g_Y\colon(D_{ij},\partial D_{ij})\mapsto (D^+_i,\partial D^+_i)$;
\item the degree of $g_Y\colon(D_{ij},\partial D_{ij})\rightarrow (D^+_i,\partial D^+_i)$
is $a_{ij}$.
\end{enumerate}
\end{n}

\begin{claim} \label{claim45}
The bundle $g_Y^*E|_{\S^2_j}$
is obtained by gluing
$D_j^-\times \T^k$ with $D_j^+\times\T^k$ along their boundaries
using the gluing map $f_j=\prod_{i=1}^{k} (\omega_i)^{a_{ij}}\colon\S^1\rightarrow \T^k$. That is, $f_j(u)=(u^{a_{1j}},\ldots ,u^{a_{kj}})$.
\end{claim}

\begin{proof} It follows from Claim~\ref{claim_43} and Property~\ref{g_property} that $g_Y^*E|_{\S^2_j}$
is obtained by identifying $\coprod_{i=1}^{k}  D_{ij}\times\T^k$
with $(\S^2_j-\bigcup_i\, int\, D_{ij})\times\T^k$ along their boundaries
(which is the union of $k$ copies of $\S^1\times\T^k$) via the gluing
maps $\omega_i^{a_{ij}}\colon\partial D_{ij}=\S^1\rightarrow \T^k$, $i=1,\ldots , k$.
(Here we are identifying $\partial D_{ij}$ with $\S^1$ using
the orientation on $\partial D_{ij}$ induced by $D_{ij}$.)

The claim now follows from the fact that the inclusion
$\S^1=\partial D^+_j\hookrightarrow D^+_j$ is a path 
in $D^+_j-\bigcup_i\, int\, D_{ij}$ which winds positively around each
$D_{ij}$ exactly once. 
\end{proof}

\begin{claim}\label{claim46}
The principal $\S^1$-bundle $A(E)|_{\S^2_j}\rightarrow \S^2_j$ is obtained by identifying
$D^-_j\times \T^k$ with $D^+_j\times\T^k$ along their boundaries
using the gluing map $f_j\colon\S^1\rightarrow \T^k$.\end{claim}

\begin{proof} By applying Proposition~\ref{prop_commute} to the inclusion map $\S^2_j\hookrightarrow Y$ we obtain 
$$
A(E)|_{\S^2_j}=A(E|_{\S^2_j})=A(E_j).
$$
This together
with Claim~\ref{claim_43} and the definition of $A(E_j)$ implies that
$A(E_j)$ is obtained by identifying
$D^-_j\times \T^k$ with $D^+_j\times\T^k$ along their boundaries
using the gluing map $A\circ\omega_j\colon\S^1\rightarrow \T^k$.
But 
$$
A(\omega_j(u))=A(1,...,1,u,1,...,1)=(u^{a_{1j}},...,u^{a_{kj}})=f_j(u).$$
\end{proof}

Lemma~\ref{lemma_key} now directly follows from Claims~\ref{claim45} and~\ref{claim46}.
\end{proof}

To finish the proof of Proposition~\ref{prop_key} we need the following lemma.

\begin{lemma}\label{lemma_pullback}
Let $E_1\to B\T^k$ and $E_2\to B\T^k$ be principal $\T^k$-bundles. Let $Z$ be a space and let $h\colon Z\rightarrow B\T^k$ be a map.
Assume that $h^*\colon H^2(B\T^k;\Z)\rightarrow H^2(Z;\Z)$ is injective. Then $h^*E_1\cong h^*E_2$ implies $E_1\cong E_2$.
\end{lemma}
\begin{proof}
 Recall that by (\ref{sec_sec2}.\ref{fact25}) $B\T^k=(\C P^\infty)^k$.
Hence $B\T^k$ is an Eilenberg-MacLane space of type
$(\Z^k,2)$, \ie $\pi_2B\T^k=\Z^k$ and $\pi_iB\T^k=0$, $i\neq 2$.
{ Therefore, we have that for any space $X$ the group $[X,B\T^k]$ of homotopy classes of maps from $X$ to (the Eilenberg-MacLane space) $B\T^k$ is isomorphic to $H^2(X;\Z^k)$~\cite[Theorem 4.57]{Hat}. This group splits naturally as follows
\begin{equation}\label{splitting}
H^2(X;\Z^k)\cong H^2(X;\Z)\oplus...\oplus H^2(X;\Z).
\end{equation}
Indeed, the splitting $\Z^k=\Z\oplus\Z\oplus\ldots\oplus\Z$ induces a natural splitting of the cochain complex $C^*(X;\Z^k)\cong C^*(X,\Z)\oplus\ldots\oplus C^*(X,\Z)$ and~(\ref{splitting}) follows directly from the definition of cohomology. 
}

 Let $h_i\colon B\T^k\rightarrow B\T^k$ classify
$E_i$. Then $h_i\circ h\colon Z\rightarrow B\T^k$ classifies
$h^*E_i$.
But the map $h^*\colon [B\T^k,B\T^k]\rightarrow [Z,
B\T^k]$, $f\mapsto f\circ h$ is the map
$h^*\colon H^2(B\T^k;\Z^k)\rightarrow H^2(Z;\Z^k)$. This map is injective because $h^*\colon H^2(B\T^k;\Z)\rightarrow H^2(Z;\Z)$ is injective and the splitting~(\ref{splitting}) is natural.
Therefore $h_1\circ h\simeq h_2\circ h$ implies
$h_1\simeq h_2$. 
\end{proof}

By the Cellular Approximation Theorem~\cite[Theorem 4.8]{Hat}, we can assume that $g\colon B\T^k\rightarrow
B\T^k$ is a cellular map. Hence $g$ restricts to the 2-skeleton $Y$. 

 Let $\iota\colon Y\rightarrow B\T^k$ be the inclusion map. Note that 
 $$A(E)=A(\iota^*E\T^k)\cong \iota^*A(E\T^k),
 $$
where the last equivalence is by Proposition~\ref{prop_commute}. Also note that
$$
(g|_Y)^*E=\iota^*g^*(E\T^k).
$$
 By Lemma~\ref{lemma_key}, $A(E)\cong (g|_Y)^*E$. Hence, $\iota^*A(E\T^k)\cong \iota^*g^*(E\T^k)$. Now, because $\iota^*$ is an isomorphism, Lemma~\ref{lemma_pullback} applies and we conclude that $A(E\T^k)\cong g^*(E\T^k)$. 
  This completes the proof of Proposition~\ref{prop_key} .
\end{proof}

The following is a natural question: given a homomorphism $A\colon\Z^k\rightarrow
\Z^k$, is there a map $f\colon B\T^k\rightarrow B\T^k$ such that $\pi_2(f)=A$?
It is well known that the answer to this question is affirmative.
Moreover, the map $f$ is unique up to homotopy. The next lemma
is  a bit more general, and will be needed later.

\begin{lemma} 
\label{lemma_h_exist}
Let $X$ be a simply connected space
and let $A\colon\pi_2 X\rightarrow \Z^k=\pi_2B\T^k$ be a homomorphism. Then there is
a unique up to homotopy $f\colon X\rightarrow B\T^k$ with $\pi_2(f)=A$.
\end{lemma}

\begin{proof}
We can { equip $X$ with a CW complex structure so that $X$ has no 1-cells~\cite[Corollary 4.16]{Hat}}.
By a simple argument we can define $f$ on the 3-skeleton
of $X$ so that $\pi_2(f)=A$ (see~\cite[Lemma 4.31]{Hat}). And since $\pi_iB\T^k=0$, $i>2$,
obstruction arguments show that $f$ can be extended
cell by cell to the whole of $X$. The proof of the uniqueness
up to homotopy is similar.
\end{proof}

\section{Principal $\T^k$-Bundles that admit $A$-Maps}

Let $E\rightarrow X$ be a principal $\T^k$-bundle and $f\colon X\rightarrow X$. Also let $A\in SL(k,\Z)$.
In this section we answer the following question:

\begin{q} \label{q51}When does there exist an $A$-map $E\rightarrow
E$ covering $f$?
$$
\xymatrixcolsep{3pc}
\xymatrix{
E\ar^{A-{\mbox{\tiny map}}}[r]\ar[d] & E\ar[d] \\
X\ar^f[r] & X
}
$$
\end{q}

Recall that by~(\ref{sec_sec2}.\ref{fact23}) every principal $\T^k$-bundle over $X$ is equivalent (as principal
bundle) to the pull-back $h^*E\T^k$ for some $h\colon X\to B\T^k$. We will use the  following notation:
$$E_h\stackrel{\textup{def}}{=}h^*E\T^k.$$ 
The next result answers Question~\ref{q51}.
It gives a relationship between $A$, $f$ and $E=E_h$
which is equivalent to the existence of an $A$-map
$E\rightarrow E $ covering $f$. The map $\rho_A$, characterized by equation~(\ref{eq_rho}), appears in
the next theorem.

\begin{theorem} 
\label{thm_Amap}
Let $A\in SL(k,\Z)$, let $X$ be a space and let $f\colon X\rightarrow X$ be a map. Also let $h\colon X\rightarrow
B\T^k$. Then
there exists an $A$-map $E_h\rightarrow
E_h$ covering $f$
$$
\xymatrixcolsep{3pc}
\xymatrix{
E_h \ar[d]\ar^{A-{\mbox{\tiny map}}}[r]& E_h \ar[d]\\
X \ar^f[r]  & X
}
$$
if and only if \,\,\,$h\circ f\simeq\rho_A\circ h$. That is,
the following diagram homotopy commutes
$$
\xymatrix{
X\ar[r]^f\ar[d]_h & X\ar[d]^h\\
B\T^k\ar[r]^{\rho_A} & B\T^k
}
$$
\end{theorem}

\begin{proof}
First suppose that there exists an $A$-map $E_h\rightarrow E_h$
covering $f$. We have the following diagram
$$
\xymatrixcolsep{4pc}
\xymatrix{
A(E_h)\ar^{A^{-1}-{\mbox{{\tiny map}}}}[r]\ar[d] & E_h \ar[d]\ar^{A-{\mbox{\tiny map}}}[r]& E_h \ar[d]\\
X\ar^{id_X}[r] &X \ar^f[r]  & X
}
$$
where the first square comes from Corollary~\ref{cor_FA_inverse} (by taking $A^{-1}$ instead of
$A$). By composing the consecutive horizontal arrows
and using Remark~\ref{rem_equivalence} we obtain a principal $\T^k$-bundle map $A(E_h)\rightarrow
E_h$ covering $f$. Therefore 
\begin{equation}
A(E_h)\cong f^*E_h\label{eq_51}
\end{equation}
and, using Proposition~\ref{prop_commute} we obtain the following equivalences
$$(\rho_A\circ h)^*E\T^k=h^*(\rho_A^*(E\T^k))
\stackrel{(\ref{eq_rho})}{\cong}h^*A(E\T^k)\stackrel{\ref{prop_commute}}{=}A(E_h)
\stackrel{(\ref{eq_51})}{\cong} f^*(E_h)=(h\circ f)^*E\T^k$$
 and it follows that $\rho_A\circ h\simeq h\circ f$.

Conversely, suppose 
\begin{equation}
\label{eq_52}
\rho_A\circ h\simeq h\circ f
\end{equation}
Then
$$A(E_h)=A(h^*E\T^k)\stackrel{\ref{prop_commute}}{=}h^*A(E\T^k)\stackrel{(\ref{eq_rho})}{\cong}h^*(\rho_A^*(E\T^k))=(\rho_A\circ h)^*E\T^k$$ $$\stackrel{(\ref{eq_52})}{\cong} (h\circ f)^*E\T^k=f^*E_h.$$
Therefore there is a principal bundle equivalence
between $A(E_h)$ and $f^*E_h$, that is there is a $id_{\T^k}$-map
$A(E_h)\rightarrow f^*E_h$ covering $id_X$. This gives
the second square in the diagram
$$
\xymatrixcolsep{4pc}
\xymatrix{
E_h\ar[d]\ar^{A-{\mbox{{\tiny map}}}}[r] &  A(E_h)\ar[d]\ar^{{id_{\T^k}-{\mbox{{\tiny map}}}}}[r]   &  f^*(E_h)\ar[d]\ar^{{id_{\T^k}-{\mbox{{\tiny map}}}}}[r]   &  E_h\ar[d] \\
X \ar^{id_X}[r]  &   X   \ar^{id_X}[r]    &   X   \ar^f[r]      &   X  
}
$$
The first square comes from Proposition~\ref{prop_FA} and the third one
from the definition of pull-back bundle.
By composing the consecutive horizontal arrows
and using Remark~\ref{rem_composition} we obtain an
$A$-map $E_h\rightarrow E_h$ covering $f$. This completes the proof of the theorem.
\end{proof}

Our next result says that
to verify condition $\rho_A\circ h\simeq h\circ f$ in the
theorem above it is enough to verify it algebraically at
the $H^2$ level.

\begin{prop} 
\label{prop_commute_equiv}
The following are equivalent
\begin{enumerate}
\item[(1)] $\rho_A\circ h\simeq h\circ f$ 
\item[(2)] $H^2(h)\circ H^2(\rho_A)=H^2(f)\circ H^2(h)$.
\end{enumerate}
Moreover, if $X$ is simply connected and $H_2(X)$ is free
then (1) and (2) are equivalent to
\begin{enumerate}
\item[(3)] $H_2(\rho_A)\circ H_2(h)=H_2(h)\circ H_2(f)$.
\item[(4)] $\pi_2(\rho_A)\circ \pi_2(h)=\pi_2(h)\circ \pi_2(f)$.
\end{enumerate}
\end{prop}

This proposition follows  from the 
following lemma.

\begin{lemma}
Let $X$ be a space and let
$\phi,\, \psi:X\rightarrow B\T^k$ be maps. 
Then the following are equivalent
\begin{enumerate}
\item[(1)] $\phi\simeq \psi$,
\item[(2)] $H^2(\phi)= H^2(\psi)$.
\end{enumerate}
Moreover, if $X$ is simply connected and $H_2(X)$ is free
then (1) and (2) are equivalent to
\begin{enumerate}
\item[(3)] $H_2(\phi)= H_2(\psi)$,
\item[(4)] $\pi_2(\phi)= \pi_2(\psi)$.
\end{enumerate}
\end{lemma}

\begin{proof} Clearly (1) implies (2). { Recall the splitting~(\ref{splitting}) from Lemma~\ref{lemma_pullback}.}
Assume $H^2(\phi)= H^2(\psi)$, then, by naturality of the splitting~(\ref{splitting}), the induced maps on the cohomology with $\Z^k$ coefficients also coincide. Now recall that the map
$H^2(\phi;\Z^k)\colon H^2(B\T^k;\Z^k)\rightarrow H^2(X;\Z^k)$ coincides
with the map $\phi^*\colon[B\T^k,B\T^k]\rightarrow
[X,B\T^k]$ given by $[\lambda]\mapsto [\lambda\circ \phi]$.
Similarly for $H^2(\psi;\Z^k)$. Hence
$\lambda\circ \phi\simeq\lambda\circ \psi$ for every
$\lambda$. Taking $\lambda=id_{B\T^k}$ we obtain
$\phi\simeq \psi$. This proves that~(2) implies~(1).

If $X$ is simply connected and $H_2(X)$ is free
then $H^2(X)\cong H_2(X)\cong \pi_2(X)$. { (The first isomorphism is by the Universal Coefficients Theorem~\cite[Theorem 3.2]{Hat} and the second one is by the Hurewicz Theorem~\cite[Theorem 4.37]{Hat})}
Therefore,
$H^2(\phi)\cong H_2(\phi)^T\cong \pi_2(\phi)^T$
(the superscript $^T$ denotes the transpose).
\end{proof}
To prove the proposition apply the above lemma to $\phi=\rho_A\circ h$
and $\psi= h\circ f$.

\section{Simply Connected Principal $\T^k$-Bundles}

Let $E\rightarrow X$ be a principal $\T^k$-bundle. Recall that $E\cong E_h=h^*E\T^k$, where the map
$h\colon X\rightarrow B\T^k$ is unique up to homotopy.
In this section we deal with the following question:

\begin{q} 
\label{q_simplyc}
When is the total space $E_h$ simply connected?
\end{q}

Note that the fundamental group of the total space $E_h$ surjects onto the fundamental group of $X$. Therefore $X$ has to be simply connected.
The next result answers Question~\ref{q_simplyc} when $X$ is simply connected.

\begin{prop} 
\label{prop_sc}
Let $X$ be a simply connected space, and let $h\colon X\rightarrow B\T^k$ be a map. Then 
the following are equivalent:
\begin{enumerate}
\item[(1)] the total space $E_h$ is simply connected;
\item[(2)] the homomorphism $\pi_2 (h)\colon\pi_2 X\rightarrow \pi_2B\T^k$ is onto;
\item[(3)] the homomorphism $H_2 (h)\colon H_2 X\rightarrow H_2B\T^k$ is onto.
\end{enumerate}
\end{prop}

\begin{proof} From the homotopy exact sequence
of the $\T^k$-bundle $\T^k\rightarrow E_h\rightarrow X$
and the fact that $\pi_1 X=0$ we obtain the exact sequence
$$\rightarrow \pi_2 X\stackrel{\partial}{\rightarrow} \pi_1 \T^k
\rightarrow \pi_1E_h\rightarrow 0$$

Therefore $\pi_1E_h=0$ if and only if $\partial$ is onto.
On the other hand from the homotopy exact sequence
of the $\T^k$-bundle $\T^k\rightarrow E\T^k\rightarrow B\T^k$
and the fact that $E\T^k$  is contractible we obtain that
$$ \pi_2 B\T^k\stackrel{\partial'}{\longrightarrow} \pi_1 \T^k
$$
 is an isomorphism. Then the equivalence (1)$\Leftrightarrow$(2)
follows from the following claim.

\begin{claim} The following diagram commutes
$$
\xymatrix{
\pi_2X \ar_{\pi_2(h)}[rd]\ar^\partial[rr] & & \pi_1\T^k \\
 & \pi_2B\T^k\ar_{\partial'}[ru] &
}
$$
\end{claim}
The claim follows from the naturality of the homotopy exact sequence
of a pair and the definition of the boundary map.

The equivalence (2)$\Leftrightarrow$(3) follows from the
naturality of the Hurewicz map and Hurewicz Theorem.
This proves the proposition. 
\end{proof}

\section{The construction}\label{sec_construction}
We specify to the case where $X$ is a simply connected 4-manifold and $f\colon X\to X$ is a diffeomorphism. We make the following collection of assumptions $(*)$.\\

\begin{enumerate}[$(*${\ref{sec_construction}.}1$)$]

\item \label{property_71}Second homotopy group $\pi_2(X)$ is a free abelian group on $m$ generators.

\item \label{property_72} The group $\pi_2(X)$ splits as a direct sum $\Z^k\oplus\Z^{m-k}$ in such a way that the first summand is $\pi_2(f)$-invariant, i.e., $\pi_2(f)|_{\Z^k}$ is an automorphism of $\Z^k\subset\pi_2(X)$.

\item \label{property_73} Let $A\in SL(k,\Z)$ be the matrix that represents $\pi_2(f)|_{\Z^k}$. Then $A$ also represents an automorphism $\R^k\to\R^k$. Assume that there exists an $A$-invariant splitting $\R^k=E_A^s\oplus E_A^c\oplus E_A^u$ and a Riemannian metric $\|\cdot\|$ on $X$ such that the numbers\\
$$
\lambda_\sigma=\min_{\substack{v\in E_A^\sigma,\\ \|v\|=1}}\|Av\|,\,\, \,\,\,
\mu_\sigma=\max_{\substack{v\in E_A^\sigma,\\ \|v\|=1}}\|Av\|, \,\,\sigma=s,c,u,
$$\\
satisfy the following inequalities
\begin{multline*}
\hfill\shoveright{\lambda_s\le\mu_s<\lambda_c\le\mu_c<\lambda_u\le \mu_u,}\hfill\\
\mu_s<m(f),\\
\hfill\lambda_u>\|Df\|,\hfill
\end{multline*}
where  $m(f)$ is minimum of the conorm $m(Df_x)$, i.e.,
$$ m(f)=\min_{\substack{v\in TX,\\ \|v\|=1}}\|Df(v)\| $$
and $\|Df\|$ is the maximum of the norm $\|Df_x\|$, i.e., 
$$\|Df\|=\max_{\substack{v\in TX,\\ \|v\|=1}}\|Df(v)\|.$$
\end{enumerate}

\begin{remark}
We allow $E_A^c$ to be trivial.
\end{remark}


\begin{theorem}\label{thm_partial_hyperbolicity}
Let $X$ be a simply connected closed 4-manifold, let  $f\colon X\to X$ be a diffeomorphism that satisfies~$(*)$ and let $\pi_h\colon E_h\to X$ be a principal $\T^k$-bundle. Assume that $E_h$ admits an $A$-map $F\colon E_h\to E_h$. Then $F\colon E_h\to E_h$ is a partially hyperbolic diffeomorphism.
\end{theorem}
 Clearly the splitting of $(*$\ref{sec_construction}.\ref{property_73}$)$ descends to a $\T^k$-invariant splitting of the tangent bundle $T\T^k=E^s_A\oplus E^c_A\oplus E^u_A$. Then the action of $\T^k$ on $E_h$ induces a $\T^k$ invariant splitting of $T\T^k=E^s\oplus E^c\oplus E^u$; here, abusing notation, $T\T^k$ is the subbundle of $TE_h$ that consists of vectors tangent to the torus fibers. Because $F\colon E_h\to E_h$ is an $A$-map, this splitting is $DF$-invariant.
\begin{add}[to Theorem~\ref{thm_partial_hyperbolicity}] The subbundles $E^s$ and $E^u$ defined above are the stable and the unstable subbundles for $F$, respectively. The center subbundle for $F$ has the form $E^c\oplus H'$, where $H'$ is a certain subbundle complementary to $T\T^k$.
\end{add}
\begin{proof}
We equip $TE_h$ with a Riemannian metric in the following way. The flat metric on the torus induces a metric on $T\T^k$. Also recall that by $(*$\ref{sec_construction}.\ref{property_73}$)$ we have equipped $X$ with a Riemannian metric $\|\cdot\|$. Choose a continuous horizontal subbundle $H\subset TE_h$ such that $TE_h=T\T^k\oplus H$. Then
$$
(D\pi_h)_x\colon H(x)\to T_{\pi_h(x)}X
$$
is an isomorphism for every $x\in E_h$. Set
$$
\|v\|=\|D\pi_h(v)\|
$$
for $v\in H$. Then extend the Riemannian metric $\|\cdot\|$ to the rest of $TE_h$ by declaring $T\T^k$ and $H$ perpendicular.

Consider the following commutative diagram
\begin{equation}\label{eq_exact_seq}
\xymatrix{
0 \ar[r]      & E^s \ar[r]\ar[d]_{DF|_{E^s}}    & TE_h  \ar[r] \ar[d]_{DF}    & E^c\oplus E^u\oplus H \ar[r]    \ar[d]_{DF\circ p} & 0  \\
0 \ar[r]      & E^s \ar[r]    & TE_h  \ar[r]     & E^c\oplus E^u\oplus H \ar[r]     & 0
}
\end{equation}
The horizontal rows are short exact sequences of Riemannian vector bundles and all vertical automorphisms fiber over $f\colon X\to X$. The last vertical arrow is defined as the composition of $DF$ and the orthogonal projection $p$ on $E^c\oplus E^u\oplus H$. Note that the diagram 
\medskip
$$
\xymatrix{
E^c\oplus E^u\oplus H \ar[d]_{D\pi_h} \ar[r]^{Df\circ p} & E^c\oplus E^u\oplus H \ar[d]_{D\pi_h}\\
TX\ar[r]^{Df}\ar[r]^{Df} & TX
}
$$
\medskip
commutes and, hence, by our choice of the Riemannian metric
$$
\|DF(p(v))\|\ge\min(\lambda_c, m(f))\|v\|.
$$
\medskip
Combining with $(*$\ref{sec_construction}.\ref{property_73}$)$ we obtain we following bound on the minimum of the conorm 
\begin{equation}
\label{eq_inequality}
m(Df\circ p)>\mu_s
\end{equation}

\begin{lemma}[\cite{HPS}, Lemma 2.18]
\label{lemma_HPS}
Let
$$
\xymatrix{
0\ar[r] & E_1 \ar[r]^i \ar[d]^{T_1} & E_2 \ar[r]^j \ar[d]^{T_2} & E_3 \ar[r] \ar[d]^{T_3} & 0 \\
0\ar[r] & E_1 \ar[r]^i & E_2 \ar[r]^j & E_3 \ar[r]  & 0
}
$$
be a commutative diagram of short exact sequences of Riemannian vector bundles, all over a compact metric space $X$, where $T_i\colon E_i\to E_i$ are bundle automorphisms over the base homeomorphism $f\colon X\to X$, $i=1,2,3$. If
$$
m(T_3|_{E_3(x)})> \|T_1|_{E_1(x)}\|
$$
for all $x\in X$, then $i(E_1)$ has a unique $T_2$-invariant complement in $E_2$.
\end{lemma}

Because we have~(\ref{eq_inequality}), we can apply Lemma~\ref{lemma_HPS} to~(\ref{eq_exact_seq}) and obtain a $DF$-invariant splitting $TE_h=E^s\oplus \widehat{E^s}$. Exchange the roles of $E^s$ and $E^u$ and apply the same argument to obtain a $DF$-invariant splitting $TE_h= \widehat{E^u}\oplus E^u$. It is easy to see that $E^c\oplus E^u\subset \widehat{E^s}$ and $E^s\oplus E^c\subset\widehat{E^u}$. Let
$$
\widehat{E^c}=\widehat{E^s}\cap\widehat{E^u}.
$$
Then, clearly, we have a $DF$-invariant splitting $TE_h=E^s\oplus \widehat{E^c}\oplus E^u$. 

To see that $F$ is partially hyperbolic with respect to this splitting pick a continuous decomposition $\widehat{E^c}=E^c\oplus H'$, and define a new Riemannian metric $\|\cdot\|'$ on $TE_h$ in the same way $\|\cdot\|$ was defined, but using $H'$ instead of $H$; i.e., we declare
\begin{enumerate}
\item $\|v\|'=\|v\|$ if $v\in T\T^k$,
\item $\|v\|'=\|D\pi_h(v)\|$ if $v\in H'$,
\item $H'$ is orthogonal to $T\T^k$.
\end{enumerate}
Now partial hyperbolicity (with respect to $\|\cdot\|'$) is immediate from the inequalities of $(*$\ref{sec_construction}.\ref{property_73}$)$.
\end{proof}

\section{The base space --- the Kummer surface}\label{sec_kummer_surface}
A K3 surface is a simply connected complex surface whose canonical bundle is trivial. All K3 surfaces are pairwise diffeomorphic and have the same intersection form $2(-E_8)\oplus
3\left(
\begin{smallmatrix}
 0 & 1\\
 1 & 0
\end{smallmatrix}\right)
$.
In this section we recall Kummer's construction of the K3 surface and describe a holomorphic atlas on it.

Consider the complex torus
$$
\T^2_\C=\C^2/(\Z\oplus i\Z)^2.
$$
Also consider the involution $\iota\colon \T^2_\C\to \T^2_\C$ given by $\iota(z_1,z_2)=(-z_1,-z_2)$. It has 16 fixed points which we  call the {\it exceptional set} and which we denote by $\E(\T^2_\C)$. Note that $\T^2_\C/\iota$ is not a topological manifold because the neighborhoods of the points in the exceptional set are cones over $\R P^3$-s. Replace the neighborhoods of the points from the exceptional set with copies of $\overline{\C P}^2$ to obtain the blown up torus $\T^2_\C\#16\overline{\C P}^2$ ({ here $\overline{\C P}^2$ stands for the two (complex) dimensional complex projective with reversed orientation;} see \eg~\cite[p.~286]{Sc} for details on complex blow up). The involution $\iota$ naturally induces a holomorphic involution $\iota'$ of $\T^2_\C\#16\overline{\C P}^2$. The involution $\iota'$ fixes 16 copies of $\C P^1$. One can check that the quotient
$$
X\stackrel{\textup{def}}{=}\T^2_\C\#16\overline{\C P}^2/\iota'
$$
is a 4-dimensional manifold. This manifold is called the {\it Kummer surface}. Note that it comes with a map
\begin{equation}\label{eq_sigma}
\sigma\colon\T^2_\C\backslash\E(\T^2_\C)\to X,
\end{equation}
which is a double cover of its image $X\backslash\E(X)$, where $\E(X)$ is the {\it exceptional set} in $X$, \ie  the union of 16 copies of $\C P^1$.
One can also check that $X$ is simply connected. (See~\cite[Chapter 3.3]{Sc} for more details.) 

In fact, $X$ is a complex surface and we proceed to describe the complex structure on $X$. For any connected open set $\mathcal V$ which is disjoint from the exceptional set $\E(X)$ and whose preimage under $\sigma$ has 2 connected components, a holomorphic chart on $\T^2_\C$ for one of the connected components of $\sigma^{-1}(\mathcal V)$ induces a chart on $\mathcal V$ by composing with $\sigma$. Hence we are left to describe the charts on a neighborhood of $\E(X)$. 

Let $p\in\E(\T^2_\C)$. We identify a neighborhood of $p$ in $\T^2_\C$ with a neighborhood $\U$ of $(0,0)$ in $\C^2$. Then we blow up $p$, which amounts to replacing $\U$ with
$$
\U'=\{(z_1, z_2, \ell(z_1, z_2)): (z_1, z_2)\in \U, (z_1, z_2)\in \ell(z_1, z_2)\}.
$$
Here $\ell(z_1, z_2)$ is a complex line through $(0,0)$ and $(z_1, z_2)$. Hence, if $(z_1, z_2)\neq (0,0)$ then $\ell(z_1, z_2)=[z_1 : z_2]$ in homogeneous coordinates. Finally, note that 
\begin{equation}\label{eq_u_2primes}
\U''=\{(z_1, z_2, \ell(z_1, z_2))\in\U'\}/(z_1, z_2, \ell(z_1, z_2))\sim(-z_1, -z_2, \ell(z_1, z_2))
\end{equation}
is identified with a neighborhood of $\C P^1\subset\E(X)$ in $X$. We will cover $\U''$ by two charts.

Note that the inclusion $\U\hookrightarrow\C^2$ induces the inclusion $\U'\hookrightarrow\C^2\#\overline{\C P}^2$ and then the inclusion $\U''\hookrightarrow \C^2\#\overline{\C P}^2/\iota''$, where $\iota''$ is induced by $(z_1,z_2)\mapsto (-z_1, -z_2)$. We will define charts for $\C^2\#\overline{\C P}^2/\iota''$. Then to obtain charts for $\U''$ one just needs to take the restrictions of the charts for $\C^2\#\overline{\C P}^2/\iota''$.

First note that
$$
\C^2\#\overline{\C P}^2=\{(z_1, z_2, \ell(z_1, z_2)): (z_1, z_2)\in \C^2, (z_1, z_2)\in \ell(z_1, z_2)\}\subset \C^2\times\C P^1.
$$
The projective line $\C P^1$ can be covered by two charts $u\mapsto [u : 1]$ and $u'\mapsto [1 : u']$. These charts extend to charts for $\C^2\#\overline{\C P}^2$ as follows
$$
\varphi_1\colon (u_1, u_2)\mapsto (u_1u_2, u_2, [u_1 : 1])
$$
and
$$
\varphi_2\colon (u_1', u_2')\mapsto (u_2', u_1'u_2', [1 : u_1']).
$$
Define $\xi\colon\C^2\to\C^2$ by $\xi(u_1, u_2)=(u_1, u_2^2)$. By a direct check, we see that the following composition
$$
\xymatrix{
\C^2\ar^{\xi^{-1}}[r] & \C^2 \ar^{\varphi_i\;\;\;\;\;\;\;}[r] &\C^2\#\overline{\C P}^2 \ar[r] & \C^2\#\overline{\C P}^2/\iota''
}
$$
is independent of the branch of $\xi$ and gives a well defined chart $\psi_i$ (homeomorphism on the image), $i=1,2$. It is also easy to see that the images of $\psi_1$ and $\psi_2$ cover $\C^2\#\overline{\C P}^2$. Calculating
\begin{equation}
\label{eq_transition}
\psi_2^{-1}\circ\psi_1(v,w)=(1/v, v^2w)
\end{equation}
confirms that the atlas is holomorphic.
\begin{remark}\label{rem_atlas}
Formulas
$$
\psi_1(v,w)=(v\sqrt{w},\sqrt{w}, [v : 1]);\;\;\;\;\;\;\;\; \psi_2(v,w)=(\sqrt{w}, v\sqrt{w}, [1 : v])
$$
also show that charts $\psi_1$ and $\psi_2$ are compatible with the charts induced from $\C^2\backslash\{(0,0)\}$ by the double cover $\C^2\backslash\{(0,0)\}\to \left(\C^2\#\overline{\C P}^2/\iota\right)\backslash\C P^1$ of the complement of the exceptional set.
\end{remark}
\begin{remark}
\label{rem_form} Consider the  2-form $dz_1\wedge dz_2$ on $\T^2_\C$ and its pushforward $\eta=\sigma_*(dz_1\wedge dz_2)$ to $X\backslash\E(X)$ (it is well defined because $dz_1\wedge dz_2=(-dz_1)\wedge (-dz_2)$). Calculating the latter in the chart $\psi_1$ yields
$$
d(v\sqrt w)\wedge d\sqrt w=\frac12dv\wedge dw.
$$
Together with an analogous calculation in the chart $\psi_2$ this implies that $\eta$ extends to a non-vanishing 2-form on $X$.
\end{remark}
Remark~\ref{rem_atlas} shows that the charts defined above for $\E(X)$ are compatible with the charts induced by $\sigma$ from the charts for $\T^2_\C$. Hence we have equipped $X$ with a holomorphic atlas.

\section{The base dynamics --- automorphisms of Kummer surfaces}

Let $B\in SL(2,\Z)$ be a hyperbolic matrix. Then $B$ induces an automorphism $B_\C\colon\T^2_\C\to \T^2_\C$. Note that after appropriately identifying $\T^2_\C$ with the real torus $\T^4$ the matrix that represents $B_\C$ is 
$
B\oplus B=
\left(
\begin{smallmatrix}
 B & 0\\
 0 & B
\end{smallmatrix}\right)
$.
We use this identification $\T^2_\C\cong\T^4$ repeatedly in what follows. The automorphism $B_\C$ naturally induces an automorphism of $\T^2_\C\#16\overline{\C P}^2$ and, hence, because the latter commutes with $\iota'$, descends to a homeomorphism $f_B\colon X\to X$. It is easy to verify that $f_B$ is, in fact, a complex automorphism of $X$.
The second integral cohomology group of $X$ is $\Z^{22}$ and the second rational cohomology group admits a splitting
\begin{equation}
\label{eq_splitting_cohomology}
H^2(X;\Q)\cong \Q^6\oplus\Q^{16},
\end{equation}
where $\Q^6$ is inherited from $H^2(\T^2_\C;\Q)$ and the remaining 16 copies of $\Q$ come from the 16 copies of ${\C P}^1$ in $\E(X)$. See~\cite[Chapter VIII]{BHPV} for a proof of these facts.
\begin{prop}\label{prop_induced_cohomology}
 The induced automorphism $f_B^*\colon H^2(X;\Z)\to H^2(X;\Z)$ is represented by the matrix $\textup{diag}(B^2,id_{\Z^4}, S_{16})$, where $S_{16}$ is a permutation matrix given by the restriction of $B_\C$ to $\E(\T^2_\C)$.
\end{prop}
\begin{proof}
Note that, by the universal coefficients theorem, it suffices to show that the induced automorphism of the rational cohomology $f^*_B\colon H^2(X;\Q)\to H^2(X;\Q)$ has the posited form. Then we can use naturality of the isomorphism~(\ref{eq_splitting_cohomology}). Under this isomorphism 
 the restriction $f^*_B|_{H^2(\T^2_\C;\Q)}$ corresponds to $B_\C^*\colon H^2(\T^2_\C;\Q)\to H^2(\T^2_\C;\Q)$ given by $(B\oplus B)\wedge(B\oplus B)$. And the restriction $f^*_B|_{\Q^{16}}$ permutes the coordinates according to the permutation $S_{16}$ given by the restriction of $B_\C$ to $\E(\T^2_\C)$. After an (integral) change of basis we obtain that $f_B^*$ is given by $\textup{diag}(B^2,id_{\Z^4}, S_{16})$.
 \end{proof}
\begin{remark}
Note that the basis in which the automorphism has the above diagonal form is not completely canonical because we use the eigenvectors that correspond to unit eigenvalues to write $(B\oplus B)\wedge(B\oplus B)$ as $\textup{diag}(B^2, id_{\Z^4})$.
\end{remark}

The goal now is to perturb $f_B$ so that the perturbation satisfies the collection of assumptions $(*)$ from Section~\ref{sec_construction}.

Set 
$
B=
\left(
\begin{smallmatrix}
 13 & 8\\
 8 & 5
\end{smallmatrix}\right)
$. Note that because of this choice of $B$ the automorphism $B_\C$ fixes points in $\E(\T^2_\C)$.

Embed the automorphism $B\colon\T^2\to\T^2$ into a 2-parameter family of diffeomorphisms of $\T^2$
$$
B_{\e,d}(x,y) =(13x-h_{\e,d}(x)+8y, 8x-h_{\e,d}(x)+5y), \;\;\;\e\ge 0, \;\;d\in\Z_+.
$$
Here $h_{\e,1}\colon S^1\to S^1$ is a $C^\infty$ smooth function that has the following properties:
\begin{enumerate}
 \item $h_{\e,1}(-x)=-h_{\e,1}(x)$;
 \item $\forall x\in S^1\;\;|h_{\e,1}'(x)|\le \e$;
 \item $h_{\e,1}(x)=h_{\e,1}(x+\frac12)=\e x$ for $x\in U$, where $U$ is a small symmetric neighborhood of $0\in S^1$;
\end{enumerate}
The existence of such function for sufficiently small $U$ can be seen by standard $C^\infty$-gluing techniques. 
 To define $h_{\e,d}\colon S^1\to S^1$ consider the $d$ sheeted self cover $S^1\to S^1$ given by $x\mapsto dx$ and let $h_{\e,d}$ be the lifting of $h_{\e,1}$ that fixes $0$. It is clear that $h_{\e,d}$ also satisfies properties~1 and~2 and the following variant of 3:
 \begin{enumerate}
 \item[$3'$.] $h_{\e,d}(x)=h_{\e,d}(x+\frac12)=\e x$ for $x\in U_d$, where $U_d$ is the connected component of $0\in S^1$ of the set $\{x: dx\in U\}$;
\end{enumerate}
 
 Note that $B_\C\colon \T^2_\C\to\T^2_\C$ embeds into the  2-parameter family $B_{\e, d}\oplus B_{\e, d}\colon\T^2_\C\to\T^2_\C$. (Recall that we have an identification $\T^4\cong\T^2_\C$.)
\begin{prop}
 \label{prop_does_induce}
The diffeomorphisms  $B_{\e,d}\oplus B_{\e,d}\colon\T^2_\C\to\T^2_\C$ induce volume preserving, Bernoulli, diffeomorphisms $f_{\e,d}\colon X\to X$ for sufficiently small $\e\ge 0$ and all $d\ge 1$.
\end{prop}
\begin{proof}

 It is easy to see that $B_{\e,d}\oplus B_{\e,d}$ fixes points from the finite set $\E(\T^2_\C)$ and that the differential at the points from $\E(\T^2_\C)$ are complex linear maps. Also, $B_{\e,d}(x,y)=B_{\e,d}(-x,-y)$, hence, $B_{\e,d}\oplus B_{\e,d}$ induces a diffeomorphism $f_{\e,d}\colon X\to X$. The fact that $f_{\e,d}$ is smooth boils down to a calculation in charts in the neighborhood of $\E(X)$. This is a routine calculation which we omit. 
 
 By calculating the Jacobian of $B_{\e,d}$ we see that diffeomorphism $B_{\e,d}\oplus B_{\e,d}$ preserves the volume $vol_{\T^2_\C}$ induced by the form $dz_1\wedge dz_2\wedge \overline{dz_1}\wedge \overline{dz_2}$. Remark~\ref{rem_form} implies that $vol_X=\sigma_*vol_{\T^2_\C}$ is induced by $\eta\wedge\overline\eta$ and hence is indeed a smooth volume. However it is clear from the definition that $f_{\e,d}$ preserves $vol_X$.
 
For sufficiently small $\e>0$ the diffeomorphism $B_{\e,d}\oplus B_{\e,d}$ is Anosov and, hence, Bernoulli. Because $vol_X(\E(X))=0$ the dynamical system $(f, vol_X)$ is a measure theoretic factor of $(B_{\e,d}\oplus B_{\e,d}, vol_{\T^2_\C})$ and, hence, is also Bernoulli by work of Ornstein~\cite{Orn}.
 \end{proof}

\begin{prop}
\label{prop_star_satisfied}
 For any sufficiently small $\e>0$  there exist a sufficiently large $d\ge 1$ such that the diffeomorphism $f_{\e,d}\colon X\to X$ satisfies the collection of assumptions $(*)$ from Section~\ref{sec_construction}.
\end{prop}
The proof of this proposition requires some lemmas. 

Let $C\colon \C^2\to\C^2$ be the automorphism given by $(z_1, z_2)\mapsto (\mu z_1, \mu^{-1} z_2)$, $\mu>1$, and let 
$$
C_*\colon  \C^2\#\overline{\C P}^2/\iota''\to  \C^2\#\overline{\C P}^2/\iota''
$$ 
be the automorphism induced by $C$ on the quotient of the blow up. (Recall that $\iota''$ is induced by $(z_1, z_2)\mapsto (-z_1, -z_2)$.) It is easy to see that $C_*$ leaves the projective line $\C P^1\subset \C^2\#\overline{\C P}^2/\iota''$ over $(0,0)$ invariant.
\begin{lemma}\label{lemma_cp1}
There exists a Riemannian metric $k$ on $\C^2\#\overline{\C P}^2/\iota''$ such that for any $x\in\C P^1$ and any $u\in T_x(\C^2\#\overline{\C P}^2/\iota'')$ we have
$$
\mu^{-2}\|u\|_k\le\|D_xC_*(u)\|_k\le\mu^2\|u\|_k,
$$
where $\|\cdot\|_k=\sqrt{k(\cdot,\cdot)}$.
\end{lemma}
\begin{proof}
Clearly it is enough to define $k$ on $ \C P^1\subset \C^2\#\overline{\C P}^2/\iota''$. (Then we extend it in an arbitrary way.)

Recall that in Section~\ref{sec_kummer_surface} we covered $\C^2\#\overline{\C P}^2/\iota''$ by two charts $\psi_1$ and $\psi_2$. Note that in both charts $\C P^1$ is given by $w=0$. We use Remark~\ref{rem_atlas} to calculate $C_*$ in charts
\begin{multline*}
C_{*,\psi_1}\colon(v,w)\stackrel{\psi_1}{\longrightarrow} (v\sqrt{w},\sqrt{w}, [v:1])\\
\shoveright{\stackrel{C_*}{\longrightarrow}(\mu v\sqrt{w},\mu^{-1}\sqrt{w},[\mu^2v:1])\stackrel{\psi_1^{-1}}{\longrightarrow}(\mu^2v,\mu^{-2}w)}\\
\shoveleft{C_{*,\psi_2}\colon(v,w)\stackrel{\psi_2}{\longrightarrow}(\sqrt{w},v\sqrt{w}, [1:v])}\\
\stackrel{C_*}{\longrightarrow}(\mu \sqrt{w},\mu^{-1}v\sqrt{w},[1:\mu^{-2}v])\stackrel{\psi_2^{-1}}{\longrightarrow}(\mu^{-2}v,\mu^{2}w)
\end{multline*}
Let us define a Hermitian metric in the chart $\psi_1$. Given a point $(v,0)$ define
\begin{equation}\label{eq_metric}
h_{(v,0)}=Q(v)dv\overline{dv}+Q(v)^{-1}dw\overline{dw},
\end{equation}
where
$$
Q(v)=\left(\frac{1}{1+|v|^2}\right)^2.
$$
Define a Hermitian metric in the chart $\psi_2$ by the same formula~(\ref{eq_metric}). The fact that these definitions are consistent can be seen from the following calculation that uses the transition formula~(\ref{eq_transition})
\begin{multline*}
(\psi_2^{-1}\circ\psi_1)^*h_{(v,0)}=Q(1/v)d(1/v)\overline{d(1/v)}+Q(1/v)^{-1}dv^2w\overline{dv^2w}\\
=Q(1/v)\frac{1}{|v|^4}dv\overline{dv}+Q(1/v)^{-1}|v|^4dw\overline{dw}=h_{(v,0)},
\end{multline*}
where the second equality follows from
$$
d(v^2w)=v^2dw+2vwdv=v^2dw
$$
when $w=0$; and the
last equality follows from the following identity
$$
Q(1/v)=|v|^4Q(v).
$$
Therefore, (\ref{eq_metric}) gives a well-defined Hermitian metric $h$ on $\C P^1\subset \C^2\#\overline{\C P}^2/\iota''$. Define the Riemannian metric $k$ as real part of $h$
$$
k=\frac{h+\bar h}{2}
$$
Notice that in charts $k$ is a warped product. Thus, we only need to prove the posited inequalities for the real parts of dual vectors $e_v$ and $e_w$. We check the inequality in the chart $\psi_1$. The calculation in the chart $\psi_2$ is completely analogous.
\begin{multline*}
\frac{k_{C_{*,\psi_1}(v,0)}(DC_{*,\psi_1}(e_v),DC_{*,\psi_1}(e_v))}{k_{(v,0)}(e_v,e_v)}=\frac{k_{(\mu^2v,0)}(\mu^2e_v,\mu^2e_v)}{k_{(v,0)}(e_v,e_v)}\\
\shoveright{=
\frac{Q(\mu^2v)\mu^4}{Q(v)}=\left(\frac{\mu^2+|\mu v|^2}{1+|\mu^2 v|^2}\right)^2;}\\
\shoveleft{\frac{k_{C_{*,\psi_1}(v,0)}(DC_{*,\psi_1}(e_w),DC_{*,\psi_1}(e_w))}{k_{(v,0)}(e_w,e_w)}=\frac{k_{(\mu^2v,0)}(\mu^{-2}e_w,\mu^{-2}e_w)}{k_{(v,0)}(e_w,e_w)}}\\
=
\frac{Q(\mu^2v)^{-1}\mu^{-4}}{Q(v)^{-1}}
=\left(\frac{\mu^2+|\mu v|^2}{1+|\mu^2 v|^2}\right)^{-2};
\end{multline*}
Finally, the posited inequalities follow from the following elementary estimate
$$
\mu^{-2}\le \frac{\mu^2+|\mu v|^2}{1+|\mu^2 v|^2}\le \mu^2.
$$
\end{proof}

Let $g_{\T^2_\C}=\textup{Re}(dz_1\overline{dz_1}+dz_2\overline{dz_2})$ be the standard flat metric on $\T^2_\C$  and let
$$
g_{d,\T^2_\C}=d^2g_{\T^2_\C}\;\;\;\textup{for}\;\; d\ge 1.
$$
We will write $\|\cdot\|_{d,\T^2_\C}$ for the induced norms.

Let $\lambda>1$ be the larger eigenvalue of $B$.
The following lemma follows immediately from property~2 of $h_{\e,d}$ and the definition of $B_{\e,d}$.
\begin{lemma}\label{lemma_lambda_e}
There exist a function $\lambda_\e$, $\e\ge 0$, such that $\lambda_\e\to \lambda$ as $\e\to 0$ and
$$
\lambda_\e^{-1}\|u\|_{{d, \T^2_\C}}\le\|D(B_{\e,d}\oplus B_{\e,d})(u)\|_{d,\T^2_\C}\le\lambda_\e\|u\|_{d,\T^2_\C}
$$
for all $d\ge 1$.
\end{lemma}

For each $d\ge 1$ consider the open set
$$
\U_d=\left(U_d\cup \left(U_d+\frac12\right)\right)^4\subset\T^2_\C
$$
(Recall that $U_d$ is $(const/d)$-neighborhood of $0$ in $S^1$ defined in the statement of property~$3'$ of $h_{\e,d}$.) Clearly $\U_d$ is a neighborhood of $\E(\T^2_\C)$ which has 16 connected components. We will write $\U_d(p)$ for the connected component of $p\in\E(\T^2_\C)$.
\begin{remark}\label{rem_u_isometric}
By definition the neighborhoods $(\U_d(p), g_{d,\T^2_\C})$ are all pairwise isometric for all $p\in\E(\T^2_\C)$ and $d\ge 1$.
\end{remark}

Let $\mu_\e>1$ be the larger eigenvalue of the matrix
$
\left(
\begin{smallmatrix}
 13-\e & 8\\
 8-\e & 5
\end{smallmatrix}\right).
$
We have
\begin{equation}
\label{eq_mu_lambda}
\mu_\e<\lambda\;\;\;\textup{for} \;\; \e>0.
\end{equation}
The following lemma follows immediately from our definition of $B_{\e,d}\oplus B_{\e,d}$.
\begin{lemma}
\label{lemma_linear}
Let $d\ge 1$ and let $p\in\E(\T^2_\C)$. Identify $\U_d(p)$ with a neighborhood of $(0,0)$ in $\C^2$ in the obvious way. Then the restriction 
$B_{\e,d}\oplus B_{\e,d}|_{\U_d(p)}$ is a complex-linear map, which is given by
$$
(z_1,z_2)\mapsto(\mu_\e z_1,\mu_\e^{-1}z_2)
$$
in the basis of eigenvectors.
\end{lemma}


\begin{proof}[Proof of Proposition~\ref{prop_star_satisfied}]
Start by fixing a sufficiently small $\e>0$ such that $\mu_\e\in(1,\lambda)$ and 
\begin{equation}
\label{eq_lambda_e_estimate}
\lambda_\e<\lambda^2,
\end{equation}
where $\lambda_\e$ comes from Lemma~\ref{lemma_lambda_e}. Consider diffeomorphism $B_{\e,d}\oplus B_{\e,d}$ and open sets $\U_d$ and $\U_d(p)$ as described above. By Remark~\ref{rem_u_isometric} each $(\U_d(p), g_{d,\T^2_\C})$ is isometric to $(\U, g)$, where $\U$ is a neighborhood of $(0,0)$ is $\C^2$ and $g=\textup{Re}( dz_1\overline{dz_1}+dz_2\overline{dz_2})$. Using Lemma~\ref{lemma_linear} and the fact that the basis of eigenvectors for $B_{\e,d}\oplus B_{\e,d}|_{\U_d(p)}$ is orthogonal, we can precompose with a rotation and obtain another isometric identification $\U_d(p)= \U$ under which $B_{\e,d}\oplus B_{\e,d}|_{\U_d(p)}$ becomes
$$
C\colon(z_1,z_2)\mapsto(\mu_\e z_1, \mu_\e^{-1}z_2);
$$
that is, the following diagram commutes
$$
\begin{CD}
 \U_d(p) @>{B_{\e,d}\oplus B_{\e,d}}>> \U_d(p)\\
 @| @|\\
  \U @>C>> \U
\end{CD}
$$
This diagram induces the commutative diagram
\begin{equation}
\label{eq_comm_diagram}
\begin{CD}
 \U''_d(p) @>{f_{\e,d}}>> \U''_d(p)\\
 @| @|\\
  \U'' @>C_*>> \U''
\end{CD}
\end{equation}
where $\U''$ is the quotient of the blow up~(\ref{eq_u_2primes}) and $\U''_d(p)$ are the corresponding neighborhoods of the 16 copies of $\C P^1$ in $X$.
(Note that the identification $\U''_d(p)=\U''$ is not isometric yet.)

Applying Lemma~\ref{lemma_cp1} to $C_*$ yields a Riemannian metric $k$ on a neighborhood of $\C P^1\subset \U''$. Extend $k$ to $\U''$ in an arbitrary way. By~(\ref{eq_mu_lambda}) we can pick a number $\bar\mu_\e\in(\mu_\e,\lambda)$. Then, by continuity, Lemma~\ref{lemma_cp1} implies that for a sufficiently small neighborhood $\V_1\subset \U''$ of $\C P^1$ we have
\begin{equation}
 \label{eq_k_estimate}
(\bar\mu_\e)^{-2}\|u\|_k\le \|D_xC_*(u)\|_k\le \bar\mu_\e^2\|u\|_k
\end{equation}
for $x\in\V_1$ and $u\in T_x(\U'')$.

Next choose a neighborhood $\V_2\supset\V_1$ such that the collar $\V_2\backslash\V_1$ has the following properties:
\begin{enumerate}
 \item any orbit of $C_*$ visits the collar $\V_2\backslash\V_1$ at most twice;
 \item any orbit of $C_*$ that visits $\V_1$ also visits the collar $\V_2\backslash\V_1$ exactly twice --- once when entering and once when leaving $\V_1$; in particular, for any $x\in\U''$ $(x,f(x))\notin (\V_1\times \U''\backslash \V_2)\cup (\U''\backslash \V_2\times\V_1)$.
\end{enumerate}
 Such choice of $\V_2$ is possible due to hyperbolicity of $C$. Also choose a smooth function $\rho\colon \U''\to [0,1]$ such that $\rho|_{\V_1}=1$ and $\rho|_{\U''\backslash \V_2}=0$. Define Riemannian metric $\tilde g$ on $\U''$ by
$$
\tilde g=\rho k+(1-\rho)(\sigma_\U)_*g.
$$
Here $\sigma_\U\colon \U\backslash(0,0)\to\U''$ is $(z_1,z_2)\mapsto (z_1,z_2,\ell(z_1,z_2))$. 

Finally, for each $d\ge1$ decompose $X$ as the union of 16 neighborhoods $\U''_d(p)$ and the complement $X\backslash \U''_d$ and define the sequence of Riemannian metrics
$$
g_{d,X}=
\begin{cases}
 \tilde g\;\;\;\textup{on} \;\;\U''_d(p)\\
 \sigma_*g_d\;\;\;\textup{on}\;\; X\backslash \U''_d
\end{cases}
$$
In this definition we used the identifications $\U''_d(p)=\U''$ and the push-forward $\sigma_*g_d$ by $\sigma$~(\ref{eq_sigma}) is well defined on the complement because the involution $\iota$ is an isometry of $(\T^2_\C, g_{d,\T^2_\C})$. Because $\tilde g=(\sigma_\U)_*g$ near the boundary of $\U''$ this definition, indeed, gives a smooth Riemannian metric on $X$. 

Denote by $\V_d$ the union of 16 copies of $\V_1$ in $(X,g_{d,X})$, denote by $\B_d$ the union of 16 copies of the collar $\V_2\backslash\V_1$ in $(X,g_{d,X})$ and let $\G_d=X\backslash(\V_d\cup\B_d)$.

We write $\|\cdot\|_{\e, d}$ for the norm induced by $g_{d,X}$. We have the following estimates:
\begin{enumerate}
 \item if $\{x, f_{\e, d}(x)\}\subset \G_d$ then
 $$
 \lambda_\e^{-1}\|u\|_{d, X}\le\|D_xf_{\e,d}(u)\|_{d,X}\le\lambda_\e\|u\|_{d,X};
 $$
  \item if $\{x, f_{\e, d}(x)\}\subset \V_d$ then
 $$
 \bar\mu_\e^{-2}\|u\|_{ d, X}\le\|D_xf_{\e, d}(u)\|_{d,X}\le\bar\mu_\e^2\|u\|_{d,X};
 $$
 \item otherwise
 $$
  K^{-1}\|u\|_{d,X}\le\|D_xf_{\e,d}(u)\|_{d,X}\le K\|u\|_{d,X};
  $$
\end{enumerate}
where  $K$ is a constant which is independent of $d$. Property~1 follows from Lemma~\ref{lemma_lambda_e}. Property~2 follows from~(\ref{eq_k_estimate}). Property~3 is due to the fact that in the collars both the dynamics ($C_*$) and the metric ($\tilde g$) do not depend on $d$. Properties~1 and~2 together with our choice of $\bar\mu_\e$ and~(\ref{eq_lambda_e_estimate}) imply that
$$
 \lambda^{-2}\|u\|_{d,X}<\|D_xf_{\e, d}(u)\|_{d,X}<\lambda^2\|u\|_{d,X}
$$
holds whenever $\{x,f_{\e, d}(x)\}\subset \V_d\cup\G_d$.

Hence, the only region without effective control on $Df_{\e, d}$ is $\B_d$, \ie when a point enters a collar or leaves a collar. However, by our construction the neighborhoods $\U''_d$ of 16 copies of $\C P^1$ in $X$ are nested, moreover,
$$
\bigcap_{d\ge 1}\U''_d(p)=\C P^1(p),
$$
where $\C P^1(p)$ is the projective line above $p\in\E(\T^2_\C)$. It follows that for large $d$ the neighborhood $\U''_d$ is (topologically) small and it takes a lot of time for an orbit of $f_{d,X}$ to travel from a neighborhood $\U''_d(p_1)$ to another neighborhood $\U''_d(p_2)$. When an orbit travels through a neighborhood $\U''_d(p)$ it meets $\B_d$ at most twice and the rest of the time it spends in $\V_d\cup \G_d$. Hence, when an orbit travels through a neighborhood $\U''_d(p)$ we may have only up to four iterates when the differential is pinched between $K^{-1}$ and $K$. These observations together with the standard adapted metric construction (see \eg~\cite{Math}) imply that there exists $d=d(K)$ and an adapted metric $g^{\textup{adapted}}_{d,X}$ such that
\begin{equation}
\label{eq_lambda2_pinching}
 \lambda^{-2}\|u\|<\|D_xf_{d,X}(u)\|<\lambda^2\|u\|,
\end{equation}
for all $x\in X$ and $u\in T_xX$, where $\|\cdot\|$ is the norm induced by $g^{\textup{adapted}}_{d,X}$.

We can check now that $(X,\|\cdot\|)$ and $f_{d,X}$ satisfy assumption $(*)$ of Section~\ref{sec_construction}. Indeed, $\pi_2(X)\cong H_2(X;\Z)\cong\Z^{22}$ verifying $(*$\ref{sec_construction}.\ref{property_71}$)$. By Proposition~\ref{prop_induced_cohomology}, $\pi_2(f_{d,X})=\pi_2(f_B)=(B^2, Id_{\Z^{20}})$ verifying $(*$\ref{sec_construction}.\ref{property_72}$)$ with $k\ge 2$. Finally, the inequalities of $(*$\ref{sec_construction}.\ref{property_73}$)$ also hold true because $B^2$ has eigenvalues $\lambda^{-2}, \lambda^2$ and we have verified~(\ref{eq_lambda2_pinching}).
\end{proof}

\section{Proof of the Main Theorem~\ref{theorem_main}}
\label{sec_proof}
Let $X$ be the Kummer surface and let $
B=
\left(
\begin{smallmatrix}
 13 & 8\\
 8 & 5
\end{smallmatrix}\right).
$
Then by Propositions~\ref{prop_induced_cohomology},~\ref{prop_does_induce} and~\ref{prop_star_satisfied} there exists a volume preserving, Bernoulli diffeomorphism $f\colon X\to X$ which verifies the collection of assumptions $(*)$ of Section~\ref{sec_construction}. Moreover, because $\pi_2(f)=(B^2,id_{\Z^{20}})$ by Proposition~\ref{prop_induced_cohomology}, we can take any $k$ in $[2,20]$ and the splitting $\Z^{22}=\Z^k\oplus\Z^{22-k}$  will verify~$(*$\ref{sec_construction}.\ref{property_72}$)$ and $(*$\ref{sec_construction}.\ref{property_73}$)$. The matrix $A\in SL(k,\Z)$ from $(*$\ref{sec_construction}.\ref{property_73}$)$ is given by
$$
A=diag(B^2,id_{\Z^{k-2}}).
$$

By Lemma~\ref{lemma_h_exist} there exist a map $h\colon X\to B\T^k$ such that $\pi_2(h)\colon\Z^k\oplus\Z^{m-k}\to\Z^k$ is the projection onto the first summand $\Z^k$. Let $\pi_h\colon E_h\to X$ be the pullback bundle $h^*E\T^k$. By Proposition~\ref{prop_sc} the total space $E_h$ is simply connected. 
Also consider the diagram 
$$
\xymatrix{
X\ar[r]^f\ar[d]_h & X\ar[d]^h\\
B\T^k\ar[r]^{\rho_A} & B\T^k
} 
$$
Recall that, by Proposition~\ref{prop_key}, $\pi_2(\rho_A)=A$. Together with~$(*$\ref{sec_construction}.\ref{property_72}$)$, this implies that the above diagram commutes on the level of $\pi_2$, and hence homotopy commutes by Proposition~\ref{prop_commute_equiv}. Then Theorem~\ref{thm_Amap} applies and yields an $A$-map $F\colon E_h\to E_h$. By Theorem~\ref{thm_partial_hyperbolicity} the diffeomorphism $F$ is partially hyperbolic. Because $F$ is an $A$-map over a volume preserving diffeomorphism, Fubini's Theorem implies that $F$ is also volume preserving. 

To establish ergodicity start by removing the 3-skeleton of $X$ and all its iterates under $f$. We obtain a subset $\bar X\subset X$ of full volume. Over $\bar X$ the bundle trivializes and the $A$-map $F$ takes the form
$$
F(x,y_1,y_2)=(f(x), B^2(y_1)+\alpha(x), y_2+\beta(x)),
$$
where $(y_1,y_2)\in\T^2\times\T^{k-2}=\T^k$. After making the coordinate change $(x,y_1,y_2)\mapsto (x, y_1+u(x), y_2)$, where $u(x)=(Id-B^2)^{-1}\alpha(x)$, $F$ takes the form
$$
F(x,y_1,y_2)=(f(x), B^2(y_1), y_2+\beta(x))
$$
Recall that $f$ is Bernoulli, $B^2\colon\T^2\to\T^2$ is Anosov and, hence, is also Bernoulli. Because the product of two Bernoulli automorphisms is also Bernoulli we can write
$$
F(z,y_2)=(T(z), y_2+\beta(z)),
$$
where $z=(x,y_1)$, $\beta(z)=\beta(x)$ and $T$ is Bernoulli. Note that this already solves the case $k=2$. Now consider an $F$-invariant $L_2$ function and use Fourier decomposition with respect to the $y_2$-coordinate to see that $F$ is ergodic (\ie the invariant function must be constant) if and only if the cohomological equation
$$
\xi(Tz)-\xi(z)=\beta(z)
$$
has a non-trivial solution $\xi$. Thus $F$ is ergodic if $\int\beta(z)dvol\neq 0$.

Recall that $\T^k$ acts on $E_h$ on the right by translation on the fiber. It is easy to see that $\rho\circ F$, $\rho\in\T^k$ is still an $A$-map and hence is volume preserving and partially hyperbolic. 
If $\int\beta(z)dvol\neq 0$ then consider
$$
F'=\rho\circ F,
$$
where $\rho=(0,\omega)\in\T^2\times\T^{k-2}$, $\omega\neq 0$.  In $(z,y_2)$-coordinates $F'$ takes the form
$$
F'(z,y_2)=(T(z), y_2+\beta(z)+\omega).
$$
Because $\int(\beta(z)+\omega) dvol=\omega\neq 0$ the diffeomorphism $F'$ is ergodic.

We have constructed partially hyperbolic diffeomorphisms on simply connected manifolds of dimension 6 to 26. To obtain higher dimensional examples one can couple these examples or couple them with sufficiently slow ergodic diffeomorphisms of spheres.

\section{Final remarks}
\subsection{The six dimensional example}
Note that our 6 dimensional example is in fact Bernoulli. It is also easy to see that it is stably non dynamically coherent. Indeed, a center leaf would cover $X$, hence, would be a trivial one-to-one cover and give a section of the bundle, but the bundle $E_h$ is non-trivial and, hence, does not admit sections.

\subsection{Real analytic version} We believe that our examples can be made real analytic by modifying the base diffeomorphism. More specifically one only needs to change the definition of $B_{\e,d}$ in the following way
$$
B_{\e,d}(x,y) =(13x-\e\sin(4d\pi x)+8y, 8x-\e\sin(4d\pi x)+5y), \;\;\;\e\ge 0, \;\;d\in\Z_+.
$$
One then has to work out a version of Lemma~\ref{lemma_cp1}. Note that calculations become tedious; in particular, because the cubic term of $B_{\e,d}$ at $(0,0)$ affects the dynamics on $\C P^1$.

\subsection{Bunching} By a more careful construction of the base diffeomorphism $f\colon X\to X$ one can obtain similar examples $F$ that are also $(2-\e)$-bunched; that is, for any $\e>0$ there exist a Riemannian metric $\|\cdot\|$ and $\lambda>1$ such that for any unit vectors, $v^s,v^c, v^u$ respectively in $E^s, E^c, E^u$ we have that  
\[\arraycolsep=1.4pt\def\arraystretch{1.6}
\begin{array}{rcccl}
& &\|DF(v^s)\|&\le& \lambda^{-2}\;\;\;\\
\lambda^{1+\e}\|DF(v^s)\|&<& \|DF(v^c)\|&<&\lambda^{-1-\e}\|DF(v^u)\|  \\
\lambda^2&\le &\|DF(v^u)\|&
\end{array}
\]

\subsection{2-connected example} It is easy to see from long exact sequence of the fiber bundle that, when $k=22$, our construction yields a partially diffeomorphism $F\colon E_h\to E_h$ of a simply connected, 2-connected, 26-dimensional manifold, \ie $\pi_1(E_h)=\pi_2(E_h)=0$.

\subsection{Irreducibility}
\label{sec_irreducible}
We recall a definition introduced in~\cite[Section 7]{FG}.
A partially hyperbolic diffeomorphism $F\colon N\to N$ is called {\it irreducible} if it verifies the following conditions:
\begin{enumerate}
 \item the diffeomorphism $F$ does not fiber over a (topologically) partially hyperbolic (or Anosov) diffeomorphism $\hat F\colon \hat N\to \hat N$ of a lower dimensional manifold $\hat N$; that is, one cannot find a fiber bundle $p\colon N\to \hat N$ and a (topologically) partially hyperbolic (or Anosov) diffeomorphism $\hat F\colon \hat N\to \hat N$ such that $p\circ F=\hat F\circ p$;
 \item if $F'$ is homotopic to $F$ then $F'$ also verifies 1;
 \item if $\tilde F$ is a finite cover of $F$ then $\tilde F$ also verifies 1 and 2. 
\end{enumerate}
\begin{conj}
 Our 6-dimensional example is irreducible.
\end{conj}

\subsection{A partially hyperbolic branched self-covering of $\S^3$}
Our construction can be applied to the Hopf bundle $\S^1\to\S^3\to\S^2$. Namely, consider the Latt\`es map of $\S^2$ induced by multiplication by $n$ on $\T^2$, $n\ge 2$. This is a rational map of degree $n^2$, which is self-covering outside of the ramification locus that consists of 4 points (see~\cite[$\mathsection 7$]{M} for a detailed description). Then, by working through the $A$-map machinery, one obtains a self map of $\S^3$ that covers the Latt\`es map and which is given by multiplication by $n^2$ in the $\S^1$ fibers. Further, by slowing down the Latt\`es map at the ramification points, one can obtain a partially hyperbolic branched self-covering of $\S^3$ of degree $n^4$. In fact, we can use a rational (non-Latt\`es) map of the base coming from Theorem~1 of~\cite{BE}. This map does not require further perturbation.

\end{document}